\documentclass[12pt,twoside]{amsart}
\usepackage{amssymb}
\usepackage{amsfonts}
\usepackage{tikz-cd}
\usepackage{color}
\usepackage{graphicx}
\usepackage{hyperref} 

\nonstopmode

\numberwithin{equation}{section} 
\font\tengothic=eufm10 scaled\magstep 1
\font\sevengothic=eufm7 scaled\magstep 1
\newfam\gothicfam
      \textfont\gothicfam=\tengothic
      \scriptfont\gothicfam=\sevengothic

\newtheorem{theorem}{Theorem}[section]
\newtheorem{lemma}[theorem]{Lemma}
\newtheorem{proposition}[theorem]{Proposition}

\theoremstyle{definition} 
\newtheorem{definition}[theorem]{Definition} 
\newtheorem{remark}[theorem]{Remark}
\newtheorem{example}[theorem]{Example}

\newtheorem{notation}[theorem]{Notation}

\newtheorem{notation and remark}[theorem]{Notation and Remark}

\newfont{\bg}{cmr10 scaled\magstep4}
\newcommand{\bigzerou}{\smash{\lower1.7ex\hbox{\bg 0}}}

\begin{document}

\newcommand\blfootnote[1]{%
  \begingroup
  \renewcommand\thefootnote{}\footnote{#1}%
  \addtocounter{footnote}{-1}%
  \endgroup
}

\begin{center}
{\Large \bf A binary tree of complete intersections \\[1ex] with the strong Lefschetz property}
\end{center}
\medskip\medskip\medskip

\begin{center}
{\sc Tadahito Harima} \\[1ex]
Department of Mathematics Education, Niigata University, \\
Niigata  950-2181, Japan \\
E-mail: harima@ed.niigata-u.ac.jp 
\medskip\medskip

{\sc Satoru Isogawa} \\[1ex]
National Institute of Technology, Kumamoto College, \\
Yatsushiro  866-8501, Japan \\
E-mail: isogawa@kumamoto-nct.ac.jp 
\medskip\medskip

{\sc Junzo Watanabe} \\[1ex]
Department of Mathematics, Tokai University, \\
Hiratsuka 259-1292, Japan \\
E-mail: watanabe.junzo@tokai-u.jp
\end{center}
\medskip\medskip\medskip

\begin{quote}
{\footnotesize 
{\sc Abstract.}\ 
In this paper we give a new family of  complete intersections which have the strong Lefschetz property. The family consists of (Artinian algebras defined by) ideals generated by power sum symmetric polynomials of consecutive degrees and of certain ideals naturally derived from them. This family has a structure of a binary tree and this observation is a key to prove that all members in it have the strong Lefschetz property. 
}
\end{quote}

\blfootnote{\noindent\textbf{Keywords}: Artinian complete intersection, power sum, strong Lefschetz property, 
Newton’s identity, central simple module, elementary symmetric polynomial.}
\blfootnote{\noindent \textbf{2020 Mathematics Subject Classification}: 13E10, 13C13.}

%
%

\section{Introduction}
\label{introduction}

Let $K$ be a field of characteristic zero, 
let $R=K[x_1,x_2,\ldots,x_n]$ be the polynomial ring in $n$ variables over $K$ with $\deg x_i = 1$ for each $i$, 
and let $I$ be a homogeneous ideal such that $A = R/I$ is Artinian. 
We say that 
$A= \bigoplus_{i=0}^c A_i$ has the {\em strong Lefschetz property} (SLP) 
if there exists a linear form $y\in A_1$ 
such that the multiplication map 
$\times y^d: A_i \rightarrow A_{i+d}$ has full rank 
for all $1\leq d \leq c-1$ and $0\leq i \leq c-d$. 
We might also say that $I$ has the SLP.
It is a long standing conjecture that 
every Artinian complete intersection should have the SLP, 
and this property as well as WLP has been studied by many authors
(e.g. \cite{HMMNWW}, \cite{HMNW}, \cite{I}, \cite{MN}, \cite{R}).

Put $R^{[n]}=K[x_1, x_2, \ldots, x_{n}]$.  
  The number $n$ of variables is not fixed and we regard the polynomial rings $\{R^{[n]}\}$
  as a family indexed by the positive integers
  and they are related by the natural inclusions $R^{[n]} \hookrightarrow R^{[n+1]}$.  
We consider only homogeneous ideals in $R^{[n]}$  with standard grading. 
Suppose that for each $n$ we are given a family $\mathcal F_{n}$ of ideals generated by a regular sequence of length $n$ in $R^{[n]}$. 
Put $\mathcal F = \bigcup _{n \geq 1} \mathcal F_{n}$. 
Temporarily we call $\mathcal F$ a {\em binary tree of complete intersections} 
if it satisfies the following conditions: 

\begin{itemize}
\item[(i)] 
For any $I \in  \mathcal F_{n}$ with $n \geq 2$, 
there exists $J \in  \mathcal F_{n-1}$ such that $I + (x_{n}) =  JR^{[n]} + (x_{n})$ as ideals  in $R^{[n]}$. 
(We do not require that $I+(x_n) \in  \mathcal F_{n}$.)
\item[(ii)] 
For any $I \in  \mathcal F_{n}$, if $x_n \not \in I$,  we have  $(I:x_n) \in  \mathcal F_{n}$. 
This means that $(I:x_n)$ is generated by a regular sequence of length $n$ in $R^{[n]}$, and it belongs to $\mathcal F_{n}$. 
(We allow ideals containing $x_n$ as members of $ \mathcal F_n $.)
\end{itemize}
We call $ \mathcal F$ a binary tree because for every $I \in  \mathcal F_{n}$, $n \geq 2$, 
we have a short exact sequence of Artinian $K$-algebras:
\[0 \to R^{[n]}/(I:x_n)  \mathop{\to}^{\phi} R^{[n]}/I \to R^{[n]}/(I + (x_n)) \to 0.\] 
Equivalently we have the ideal version of the exact sequence
\[0 \to \big(I:x_n\big) \mathop{\to}^{\phi}  I \to \big(I + (x_n)\big) / (x_n) \to 0.\]
In both sequences $\phi$ is defined by the multiplication by $x_{n}$. 
Thus $\mathcal F$ is a binary tree if we call $(I:x_n)$ the left child and $\big( I + (x_n) \big)/(x_n)$ the right child.  
A member $I \in  \mathcal F_n$ has the left child if and only if $x_n \not  \in I$. 
Any $I \in \mathcal F _1$ does not have a right child. 

It is easy to see that the family of monomial  complete intersections $ \mathcal F = \bigcup_{n \geq 1}  \mathcal F _n$, 
\[\ \mathcal F _n= \{(x_1^{a_1}, x_2^{a_2}, \ldots, x_n^{a_n}) \; | \; a_1 \geq 1, a_2 \geq 1, \ldots, a_n \geq 1 \}\]
is a binary tree.
For another such example, let $ \mathcal F_n$ be the family of ideals $I \subset R^{[n]}$ such that
$\textrm{{In}}(I)$ (``In'' is the initial ideal with respect to the reverse lexicographic order) is a monomial complete intersection.
Then  $ \mathcal F = \bigcup _{n \geq 1}  \mathcal F_n$ is a binary tree. 
 Indeed
$\textrm{In}(I:x_n)$ and $\textrm{In}(I+(x_n))$ are monomial complete intersections for every $I \in \mathcal F_n$ because $\textrm{In}(I)$ is a monomial complete intersection and it holds that  $\textrm{In}(I:x_n) = \textrm{In}(I):x_n$ and $\textrm{In}(I+(x_n)) = \textrm{In}(I)+(x_n)$. 
Hence it is easy to see that $ \mathcal F$ is a binary tree. Furthermore it is well known that the monomial Artinian complete intersections have the SLP (\cite{Stanley},\cite{W}). Hence by Proposition 33 of \cite{HW0} it follows that the ideals in this family have the SLP. 

It seems conceivable that  
it should be possible to prove that all members in a binary tree of complete intersections have the SLP, 
if not directly, perhaps with some additional assumptions on the family.
Indeed if we impose the following condition on $\mathcal F = \bigcup_{n \geq 1}  \mathcal F _n$  in addition to (i) and (ii) above
then by Theorem 1.2 of \cite{HW1} the conditions (i), (ii) and (iii) imply that all members in $\mathcal F$ have the SLP.  

\begin{itemize}
\item[(iii)] 
If $I \in \mathcal F_n$, then all central simple modules with respect to the image of $x_n$ in  $R^{[n]}/I$ 
are members of $\mathcal F_{n-1}$, 
that is, these modules themselves have the structure of an Artinian $K$-algebra 
and the defining ideals are members of $\mathcal F_{n-1}$.  
\end{itemize}

In this paper,  
we prove that 
the complete intersections defined by the power sum symmetric polynomials of consecutive degrees 
can be embedded in a family of ideals satisfying these three conditions. 
The notion of the central simple modules was introduced and some basic properties were studied in \cite{HW1}.
The central simple modules were used effectively 
in the papers \cite{altafi_iarrobino_khatami}, \cite{harima_wachi_watana_1}, \cite{HW2}, \cite{HW3} and \cite{I_M}.

In the case of power sums of consecutive degrees, 
it turns out that all central simple modules are principal modules, 
and furthermore each one of them is isomorphic to a complete intersection. 
The description of these complete intersections are given in Theorem~3.1.
Once Theorem 3.1 is proved, we start over with newly found complete intersections to find their central simple modules.
This procedure can be carried out to the end and as a result we can determine the family $\mathcal F$ of ideals which is a binary tree containing the ideals generated by power sums of
consecutive degrees.  This is stated in the sections 4 and 5.

As to the definition of a binary tree of ideals,  it should be noted  that we adapted the definition of the principal radical system developed by
Hochster--Eagon~\cite{hochster_eagon}. 
Their idea was used in H.\ Kleppe and D.\ Laksov~\cite{kleppe_laksov} in the same sense as the original due to Hochster--Eagon.   
The definition was slightly modified in the paper  ~\cite{mcdaniel_watana_1}.
The structure of the  binary tree of the monomial complete intersections was used
effectively in \cite{mcdaniel_watana_2}.

%
%

\section{Central simple modules and the strong Lefschetz property}
\label{survey}

We recall the notion of central simple modules 
for a standard graded Artinian $K$-algebra 
and review a sufficient condition for the SLP in terms of central simple modules. 
For details we refer the reader to \cite{HW1} and \cite{HW2}. 

\begin{definition}
Let $A$ be a standard graded Artinian $K$-algebra 
and let $y$ be a linear form of $A$. 
Let $p={\rm min} \{i\mid y^i=0\}$. 
Then we have a descending chain of ideals in $A$: 
\[
A=(0:y^p)+(y) \supset (0:y^{p-1})+(y) \supset 
\cdots \supset (0:y)+(y) \supset (y). 
\]
From among the sequence of successive quotients 
$$
\frac{(0:y^{p-i})+(y)}{(0:y^{p-i-1})+(y)}
$$
for $i=0,1,2,\ldots,p-1$, 
pick the non-zero spaces and rename them 
\[
U_1, U_2, \ldots, U_s. 
\] 
Note that $U_1=A/(0:y^{p-1})+(y)$. 
We call the graded $A$-module $U_i$ 
the {\em $i$-th central simple module of $(A,y)$}. 
\end{definition}

\begin{remark}\label{Re_csm} 
With the same notation as above, 
let $q$ be the least integer such that $((0:y^{q})+(y)) \neq (y)$. 
Then we have the following by the definition. 
\begin{itemize}
\item[(1)] 
$U_s=((0:y^{q})+(y))/(y)$. 
\item[(2)] 
If $(0:y^q)=A$, 
then $(A,y)$ has only one central simple module which is isomorphic to $A/yA$. 
Such examples are given in Lemma 6.1 of \cite{HW1} and Theorem~\ref{Th3-1}(1) stated later. 
\item[(3)] 
If $s>1$, 
then it is possible to regard $U_1, U_2,\ldots,U_{s-1}$
as the full set of the central simple modules of $(A/(0:y^q), \overline{y})$, 
where $\overline{y}$ is the image of $y$ in $A/(0:y^q)$. 
\item[(4)]
Furthermore, 
it is easy to see that every central simple module of  
$(A/(0:y^{j}), \bar{y})$ appears as a central simple  module of $(A, y)$, 
where $\bar{y}$ is the image of $y$ in  $A/(0:y^{j})$ for each $j\geq 1$.
\end{itemize}
\end{remark}

\begin{definition}
Let $A$ be a standard graded Artinian $K$-algebra 
and let $V=\bigoplus_{i=a}^{b} V_i$ be a finite graded $A$-module with $V_a\neq (0)$ and $V_b\neq (0)$. 
Then, 
we say that $V$ has the SLP as an $A$-module 
if there exists a linear form $y$ of $A$ 
such that the multiplication map $\times y^d: V_i \rightarrow V_{i+d}$ 
has full rank for all $1\leq d\leq b-a$ and $a\leq i\leq b-d$.  
\end{definition}

\begin{theorem}[\cite{HW1}, Theorem 1.2] \label{csm}
Let $A$ be a standard graded Artinian Gorenstein $K$-algebra. 
Then the following conditions are equivalent: 
\begin{itemize}
\item[\rm{(i)}] 
$A$ has the SLP. 
\item[\rm{(ii)}] 
There exists a linear form $y$ of $A$ such that 
all the central simple modules of $(A,y)$ have the SLP. 
\end{itemize}
\end{theorem}

\begin{remark}
Let $A$ be an Artinian Gorenstein $K$-algebra with the SLP. 
Then it is not always true that 
all central simple modules of $(A,y)$ have the SLP 
for all linear forms $y$ of $A$ 
(see Example 6.10 in \cite{HW1}).   
\end{remark}

%
%

\section{Complete intersections defined by power sums of consecutive degrees}
\label{consecutive}

Let $R=K[x_1,x_2,\ldots,x_n]$ be the polynomial ring over a field K of characteristic zero. 
Throughout this paper we work in the polynomial ring $\tilde{R}=R[z]$, 
where $z$ is another indeterminate. 

Define $e_i \in R$ to be the elementary symmetric polynomial of degree $i$ {\em with the sign} 
so that 
$$
\prod_{i=1}^n(z-x_i)=e_0z^n + e_1z^{n-1} + e_2z^{n-2} + \cdots + e_{n-1}z + e_n, 
$$
i.e., 
$$
e_0=1, \ e_i=(-1)^i\sum_{1\leq j_1<j_2< \cdots <j_i\leq n}x_{j_1}x_{j_2} \cdots x_{j_i}
$$
for $1\leq i \leq n$.
Let $p_i \in R$ and $\tilde{p}_i \in \tilde{R}$ be the power sums of degree $i \geq 0$, i.e., 
$$
p_i=x_1^i + x_2^i + \cdots + x_n^i, \ \ \tilde{p}_i=p_i+z^i. 
$$
\medskip

The purpose of this section is to give a proof of the following theorem. 

\begin{theorem}\label{Th3-1} 
Let $a$ be a positive integer, 
let $I$ be the ideal of $\tilde{R}$ generated by 
$$
\{\tilde{p}_a, \tilde{p}_{a+1}, \ldots, \tilde{p}_{a+n}\}
$$ 
and $\overline{z}$ the image of $z$ in the Artinian complete intersection $K$-algebra $A=\tilde{R}/I$. 
\begin{itemize}
\item[(1)] 
If $a=1$, then $(A, \bar{z})$ has only one central simple module $U_1$. 
This is given by $U_1 \cong R/{(e_1, e_2, \ldots, e_n)}$.

\item[(2)] 
Assume $a \geq 2$. Then $(A, \bar{z})$ has $n+1$ simple modules.
The j-th central simple module is given by
$$
U_j \cong \displaystyle\frac{R}{(\underbrace{p_{a-1}, p_a, \ldots, p_{a+j-3}}_{j-1}, \underbrace{e_j, e_{j+1}, \ldots, e_n}_{n-j+1})} 
$$
for $j=1,2,\ldots, n+1$.
\end{itemize}
\end{theorem}

We need some preparations for the proof of Theorem~\ref{Th3-1}.

\begin{remark}[Newton's identity]\label{Rem3-2} 
We recall Newton's identity: 
$$
ke_k=-\sum_{i=0}^{k-1} e_ip_{k-i} 
$$
for all $k\geq 1$, 
where we define $e_j=0$ if $j>n$ (see page 2 in \cite{S}). 
In particular, the following holds: 
$$
\sum_{i=0}^n e_ip_{m-i} = 0 
$$  
for all $m\geq n$. 
\end{remark}

\begin{remark}\label{Rem3-3} \;
\begin{itemize} 
\item[(1)] 
With the same notation as Theorem~\ref{Th3-1}, 
consider the ideals $\{J_j\}$ as follows: 
$$
\begin{array}{rcl}
J_1 &=& (e_1, e_2, \ldots, e_n), \\
J_j &=& (p_{a-1}, p_a, \ldots, p_{a+j-3}, e_j, e_{j+1},  \ldots, e_n) \ \ (j=2,3,\ldots,n), \\
J_{n+1} &=& (p_{a-1}, p_a, \ldots, p_{a+n-2}). \\
\end{array}
$$
We notice that these are the ideals describing $n+1$ central simple modules of $(A, \overline{z})$. 
Then, the following inclusions hold: 
$$
J_{n+1}\subset J_{n}\subset \cdots\subset J_{j+1}\subset J_j \subset\cdots\subset J_1. 
$$
Indeed, substituting $k=a+j-2$ in Newton's identity, 
it follows that 
$$
\begin{array}{rcl} 
(a+j-2)e_{a+j-2} &=&\displaystyle -\sum_{i=0}^{a+j-3}e_ip_{a+j-2-i} \\
 &=&\displaystyle -\{\sum_{i=0}^{j-1}e_ip_{a+j-2-i}+\sum_{i=j}^{a+j-3}e_ip_{a+j-2-i} \}. 
\end{array}
$$
Hence 
$$
e_0p_{a+j-2}+e_1p_{a+j-3}+\cdots +e_{j-1}p_{a-1} \in (e_j, e_{j+1},\ldots,e_n). 
$$ 
Since $e_0=1$, 
this implies that $p_{a+j-2}\in J_j$. 
Hence $J_{j+1} \subset J_j$. 
On the other hand, it is well known that $n$ power sums of consecutive degrees,  
$$
\{p_{a-1}, p_a, \ldots, p_{a+n-2}\}, 
$$
is a regular sequence in $R$ (see Lemma 7.2 in \cite{HW2} and Proposition 2.9 in \cite{CKW})). 
Hence, every set 
$$
\{p_{a-1}, p_a, \ldots, p_{a+j-3}, e_j, e_{j+1},  \ldots, e_n\}
$$
is also a regular sequence in $R$ for $j=2,3,\ldots,n$.  
In other word, every ideal $J_j$ is a complete intersection. 

\item[(2)] 
We keep the same notation as Theorem~\ref{Th3-1}. 
Then we can easily check that 
$$
\frac{(0:\overline{z}^{j+1})+(\overline{z})}{(0:\overline{z}^j)+(\overline{z})} 
\cong \frac{(I:z^{j+1})+(z)}{(I:z^j)+(z)}
$$
as $\tilde{R}$-modules for $j=0,1,2,\ldots$. 
Hence, to observe the central simple modules of $(A,\overline{z})$, we would like to discuss the behavior of ideals 
$\{(I:z^j) \mid j=0,1,2,\ldots\}$.
\end{itemize}
\end{remark}

\begin{notation} \label{def of f} 
Let 
$$
f^{(0)}=f^{(0)}(z)=e_0z^n + e_1z^{n-1} + e_2z^{n-2} + \cdots + e_{n-1}z + e_n \in \tilde{R} 
$$
and let
$$
f^{(k)}=f^{(k)}(z)=\frac{\partial^k}{\partial z^k}f^{(0)}(z)
$$
be the $k$-times partial derivation of $f^{(0)}$ 
for $k=1,2,\ldots,n$, i.e., \\
$$ 
f^{(1)}=f^{(1)}(z)=\frac{\partial}{\partial z}f^{(0)}(z)=ne_0z^{n-1}+(n-1)e_1z^{n-2}+\cdots+2e_{n-2}z+e_{n-1}, \\[2ex] 
$$
\\[-35pt]
\begin{eqnarray*}
f^{(k)}=f^{(k)}(z)=\frac{\partial}{\partial z}f^{(k-1)}(z)
&=&\{n(n-1)\cdots(n-k+1)\}e_0z^{n-k} \\[1ex ]
& & \ \ \ +\{(n-1)(n-2)\cdots(n-k)\}e_1z^{n-k-1} \\[1ex] 
& & \ \ \ \ \ +\cdots+\{k(k-1)\cdots1\}e_{n-k}. 
\end{eqnarray*}
\end{notation}

Following is a key to the proof of Theorem~\ref{Th3-1}, and we can calculate the left child and the right child of our binary tree of complete intersections by using this proposition.

\begin{proposition}\label{Prop3-5}
Fix an integer $a\geq 2$. 
With the same notation as above, 
define ideals $I_0, I_1, \ldots, I_{n}$ of $\tilde{R}$ as follows: 
$$ 
\begin{array}{rcl} 
I_0 & = & (\tilde{p}_a, \tilde{p}_{a+1}, \ldots, \tilde{p}_{a+n}), \\
I_1 & = & (\tilde{p}_a, \tilde{p}_{a+1}, \ldots, \tilde{p}_{a+n-1}, f^{(0)}), \\
I_2 & = & (\tilde{p}_a, \tilde{p}_{a+1}, \ldots, \tilde{p}_{a+n-2}, f^{(1)}, f^{(0)}), \\
I_k & = & (\tilde{p}_a, \tilde{p}_{a+1}, \ldots, \tilde{p}_{a+n-k}, f^{(k-1)}, f^{(k-2)}, \ldots, f^{(1)}, f^{(0)}) \\
\end{array}
$$
for $k=3,4,\ldots,n$.  
Then, for $k=0,1,2,\ldots,n$, 
the element $z^af^{(k)}$ can replace the element $\tilde{p}_{a+n-k}$ 
in the ideal $I_k$ as a member of the generating set, 
i.e., 
$$
\begin{array}{rcl}
I_0 &=&  (\tilde{p}_a, \tilde{p}_{a+1}, \ldots, \tilde{p}_{a+n-1}, z^af^{(0)}), \\ 
I_1 &=&  (\tilde{p}_a, \tilde{p}_{a+1}, \ldots, \tilde{p}_{a+n-2}, z^af^{(1)},f^{(0)}), \\ 
I_2 &=&  (\tilde{p}_a, \tilde{p}_{a+1}, \ldots, \tilde{p}_{a+n-3}, z^af^{(2)},f^{(1)},f^{(0)}), \\ 
I_k &=&  (\tilde{p}_a, \tilde{p}_{a+1}, \ldots, \tilde{p}_{a+n-k-1}, z^af^{(k)}, f^{(k-1)}, f^{(k-2)}, \ldots, f^{(1)}, f^{(0)}) \\
\end{array}
$$ 
for $k=3,4,\ldots,n-1$, and $I_n=(z^af^{(n)}, f^{(n-1)}, f^{(n-2)}, \ldots, f^{(1)}, f^{(0)})$. 
\end{proposition}

\begin{proof} 
We treat the three cases for $k=0,1,2$ separately. 
As one will see the proof for $k=2$ works for all other cases. 

Let $k=0$. 
We show that 
$$
\tilde{p}_{a+n}-z^af^{(0)} \in (\tilde{p}_a, \tilde{p}_{a+1}, \ldots, \tilde{p}_{a+n-1}). 
$$
Substituting $m=a+n$ in Remark~\ref{Rem3-2}\;(Newton's identity), we get
$$
\sum_{i=0}^n e_ip_{a+n-i}=0. 
$$
In this equality replace $p_{a+n-i}$ by $\tilde{p}_{a+n-i}-z^{a+n-i}$ for $i=0,1,\ldots,n$. 
Then we have 
$$
\sum_{i=0}^n e_i\tilde{p}_{a+n-i}=z^a\sum_{i=0}^ne_iz^{n-i}=z^af^{(0)}. 
$$
Hence, since $e_0=1$, 
it follows that 
$$
\tilde{p}_{a+n}-z^af^{(0)}=-\sum_{i=1}^n e_i\tilde{p}_{a+n-i}. 
$$
This proves the assertion for the case $k=0$. 

Next we prove the assertion for the case $k=1$. 
We show that 
$$
\tilde{p}_{a+n-1}-z^af^{(1)} \in (\tilde{p}_a, \tilde{p}_{a+1}, \ldots, \tilde{p}_{a+n-2}, f^{(0)}). 
$$
Since 
$$
f^{(0)}=\prod_{i=1}^n(z-x_i), f^{(1)}=\sum_{i=1}^n\left(\prod_{j\neq i}(z-x_j)\right), 
$$
it follows that 

$$
\begin{array}{rcl}
\displaystyle \frac{f^{(1)}}{f^{(0)}} & = & \displaystyle \sum_{i=1}^n\frac{1}{z-x_i} 
\ = \ \frac{1}{z}\sum_{i=1}^n\frac{1}{1-\frac{x_i}{z}} \\[2ex]  
& = & \displaystyle \frac{1}{z}\sum_{i=1}^n \left(1+\frac{x_i}{z}+\frac{{x_i}^2}{z^2}+\frac{{x_i}^3}{z^3}+\cdots \right) 
\ = \ \displaystyle \frac{1}{z}\sum_{i\geq 0}p_iz^{-i}. 
\end{array}
$$
Hence we have 
$$
z^af^{(1)} = f^{(0)}\sum_{i\geq 0}p_iz^{a-1-i} 
= f^{(0)}\left(\sum_{i=0}^{a-1}p_iz^{a-1-i}\right) +f^{(0)}\left(\sum_{i\geq a}p_iz^{a-1-i}\right).  
$$
Furthermore, by Newton's identity, 
the second term of the right hand side of this equality is actually a finite sum as follows: 
\begin{eqnarray*}
&&\displaystyle f^{(0)}\left(\sum_{i\geq a}p_iz^{a-1-i}\right)  = 
\displaystyle (e_0z^n + e_1z^{n-1} + \cdots + e_{n-1}z + e_n)\sum_{i\geq a}p_iz^{a-1-i} \\[1ex] 
& =& \displaystyle \left\{\sum_{j=0}^{n-1}\left(\sum_{s+i=a+n-1-j, \ i\geq a}e_sp_i\right)z^j \right\}
+ \displaystyle \left\{\sum_{j\geq 1}\left(\sum_{s=0}^ne_sp_{t(j,s)}\right)z^{-j} \right\}   \\[1ex]
&= &\displaystyle \left\{\sum_{j=0}^{n-1}\left(\sum_{s+i=a+n-1-j, \ i\geq a}e_sp_i\right)z^j \right\},  
\end{eqnarray*}
where $t(j,s)=a+n-1+j-s$. 
Hence, 
the second term can be expressed by the product of matrices as follows: 
$$
f^{(0)}\left(\sum_{i\geq a}p_iz^{a-1-i}\right)  = {\bf e}{\bf P}{\bf u}, 
$$
where {\bf e}, {\bf P} and {\bf u} are three matrices of the following forms 
$$
{\bf e}=[e_{n-1},  e_{n-2}, \ldots, e_1, e_0], 
{\bf P}=
\left[
\begin{array}{ccccccc}
p_a          &      &           & \bigzerou  \\
p_{a+1}     & p_a &           &            \\
\vdots     & \ddots     & \ddots &           \\
p_{a+n-1}  & \cdots      &  p_{a+1}         & p_a   \\     
\end{array}
\right], 
{\bf u}=
\left[
\begin{array}{c}
1          \\
z    \\
\vdots     \\
z^{n-1}  \\     
\end{array}
\right].
$$
Thus, we get the equality 
$$
z^af^{(1)}=f^{(0)}\phi+{\bf e}{\bf P}{\bf u}, 
$$
where $\phi=\sum_{i=0}^{a-1}p_iz^{a-1-i} \in \tilde{R}$. 
(This equality is needed in the proof of Proposition~\ref{Prop4-1} again.)
Put 
$$
\tilde{{\bf P}}=
\left[
\begin{array}{ccccccc}
\tilde{p}_a          &      &           & \bigzerou  \\
\tilde{p}_{a+1}     & \tilde{p}_a &           &            \\
\vdots     & \ddots     & \ddots &           \\
\tilde{p}_{a+n-1}  & \cdots      &  \tilde{p}_{a+1}         & \tilde{p}_a   \\     
\end{array}
\right], 
{\bf Z}=
\left[
\begin{array}{ccccccc}
1          &      &           & \bigzerou  \\
z     & 1 &           &            \\
\vdots     & \ddots     & \ddots &           \\
z^{n-1}  & \cdots      &  z        & 1   \\     
\end{array}
\right]. 
$$
Then, since ${\bf P}=\tilde{{\bf P}}-z^a{\bf Z}$, 
it follows that 
$$
z^af^{(1)}=f^{(0)}\phi+{\bf e}\tilde{{\bf P}}{\bf u}-z^a{\bf e}{\bf Z}{\bf u}. 
$$
Hence, noting that ${\bf e}{\bf Z}{\bf u}=f^{(1)}$, 
we have 
$$
2z^af^{(1)}=f^{(0)}\phi+{\bf e}\tilde{{\bf P}}{\bf u}. 
$$
This shows that $2z^af^{(1)}$ is a linear combination of 
$$
f^{(0)}, \tilde{p}_a, \tilde{p}_{a+1}, \ldots, \tilde{p}_{a+n-1}
$$
with coefficients in $\tilde{R}$. 
One sees that the coefficient of $\tilde{p}_{a+n-1}$ in the linear combination is $e_0=1$. 
This proves the assertion for the case $k=1$. 

We proceed to a proof for the case $k=2$. 
We show that 
$$
\tilde{p}_{a+n-2}-z^af^{(2)} \in (\tilde{p}_a, \tilde{p}_{a+1}, \ldots, \tilde{p}_{a+n-3}, f^{(1)}, f^{(0)}). 
$$
Differentiating both sides of the equation
$$
z^af^{(1)}=f^{(0)}\phi+{\bf e}{\bf P}{\bf u}, 
$$
we have 
$$
z^af^{(2)}+az^{a-1}f^{(1)}=f^{(1)}\phi+f^{(0)}\phi^{(1)}+{\bf e}{\bf P}{\bf u}^{(1)}, 
$$
where ${\bf u}^{(1)}$ is the transpose of the vector $[0, 1, 2z, 3z^2, \cdots, (n-1)z^{n-2}]$. 
Hence, substituting ${\bf P}=\tilde{{\bf P}}-z^a{\bf Z}$, 
we have 
$$
z^af^{(2)}+az^{a-1}f^{(1)}=f^{(1)}\phi+f^{(0)}\phi^{(1)}+{\bf e}\tilde{{\bf P}}{\bf u}^{(1)}-z^a{\bf e}{\bf Z}{\bf u}^{(1)}.  
$$
Furthermore, a straightforward calculation leads to the following equalities: 
$$
{\bf e}{\bf Z}{\bf u}^{(1)} = \frac{1}{2}f^{(2)}. 
$$
Thus we get the equality 
$$
\frac{3}{2}z^af^{(2)}=(\phi-az^{a-1})f^{(1)}+\phi^{(1)}f^{(0)}+{\bf e}\tilde{{\bf P}}{\bf u}^{(1)}.  
$$
This shows that $\frac{3}{2}z^af^{(2)}$ is a linear combination of 
$$
f^{(1)}, f^{(0)}, \tilde{p}_a, \tilde{p}_{a+1}, \ldots, \tilde{p}_{a+n-2}
$$
with coefficients in $\tilde{R}$. 
One sees that the coefficient of $\tilde{p}_{a+n-2}$ in the linear combination is $e_0=1$. 
This proves the assertion for the case $k=2$. 

Proof for $k>2$ is the same as the case $k=2$ with the obvious modification, 
except that we use the equalities stated in Lemma~\ref{Lem3-6} instead of the equality ${\bf e}{\bf Z}{\bf u}^{(1)} = \frac{1}{2}f^{(2)}$. 
\end{proof}

\begin{lemma}\label{Lem3-6}
With the same notation as the proof of Proposition~\ref{Prop3-5}, 
let ${\bf u}^{(k)}$ be the $n$ dimensional column vector as follows: 
$$
{\bf u}^{(k)} 
={}^{t}\!\left[\frac{\partial^k}{\partial z^k} 1, \frac{\partial^k}{\partial z^k} z,  \frac{\partial^k}{\partial z^k} z^2, 
\ldots, \frac{\partial^k}{\partial z^k} z^{n-1} \right], 
$$
where we denote the transpose of a vector by ${}^{t}[\ ]$. 
Then the following holds: 
$$
{\bf e}{\bf Z}\{{\bf u}^{(k)}\} = \frac{1}{k+1}f^{(k+1)}
$$
for $k=2,3,\ldots,n-1$. 
\end{lemma}

\begin{proof}
Since 
$
f^{(k+1)}=\frac{\partial^k}{\partial z^k} f^{(1)}
           =\frac{\partial^k}{\partial z^k} ({\bf e}{\bf Z}{\bf u})
           ={\bf e}\frac{\partial^k}{\partial z^k}({\bf Z}{\bf u}),
$
it is enough to show that
$$
\frac{\partial^k}{\partial z^k}({\bf Z}{\bf u})=(k+1){\bf Z}\{{\bf u}^{(k)}\}.
$$
Noting that 
$ 
{\bf Z}{\bf u}
  ={}^{t}\! \left[ \frac{\partial}{\partial z} z, \frac{\partial}{\partial z} z^2, 
                   \ldots, 
                   \frac{\partial}{\partial z} z^{n} 
                   \right]
$, we have
$$
\left[\textrm{the}\;
i\textrm{-th component of }
\frac{\partial^k}{\partial z^k}({\bf Z}{\bf u})
\right]
=\frac{\partial^{k+1}}{\partial z^{k+1}}(z^i)
=i^{\underline{k+1}}z^{i-k-1},
$$
where we denote 
$i^{\underline{k+1}}=i(i-1)\cdots(i-k)$
the falling factorial. 
On the other hand,
\begin{eqnarray*}
\left[\textrm{the}\;
i\textrm{-th component of }
(k+1){\bf Z}\{{\bf u}^{(k))}\}
\right] 
&=&(k+1)\sum_{j=1}^{i}z^{i-j}\frac{\partial^{k}}{\partial z^{k}}(z^{j-1}) \\
=(k+1)\sum_{j=1}^{i}z^{i-j}(j-1)^{\underline{k}}z^{j-k-1} 
&=&\left((k+1)\sum_{j=1}^{i}(j-1)^{\underline{k}}\right)z^{i-k-1} \\
&=&i^{\underline{k+1}}z^{i-k-1}.
\end{eqnarray*}
Hence the assertion follows. In the last equality of the above equation, we use the following identity
$$
(k+1)\sum_{j=1}^{i}(j-1)^{\underline{k}}
=\sum_{j=1}^i (j^{\underline{k+1}}-(j-1)^{\underline{k+1}})
=i^{\underline{k+1}}\; .
$$
\end{proof}

\begin{lemma}\label{KeyLem1} 
With the same notation as Theorem~\ref{Th3-1}, 
define the ideals $\mathfrak{a}_0, \mathfrak{a}_1, \mathfrak{a}_2, \ldots, \mathfrak{a}_{n+1}$ of $\tilde{R}$ as follows: 
$$
\begin{array}{rcl}
\mathfrak{a}_0 &=& (p_a,p_{a+1},\ldots,p_{a+n-1},z),  \\
\mathfrak{a}_1 &=& (p_a,p_{a+1},\ldots,p_{a+n-2},e_n,z),  \\
\mathfrak{a}_2 &=& (p_a,p_{a+1},\ldots,p_{a+n-3},e_{n-1},e_n,z),  \\
\mathfrak{a}_k &=& (p_a,p_{a+1},\ldots,p_{a+n-k-1},e_{n+1-k},e_{n+2-k},\ldots,e_n,z)
\end{array}
$$ 
for $k=3,4,\ldots,n-1$, $\mathfrak{a}_{n} = (e_1,e_2,\ldots,e_n,z)$ and $\mathfrak{a}_{n+1}=\tilde{R}$. 

\begin{itemize}
\item[(1)] For $k=0,1,2,\ldots,n$, 
we have the following equalities: 
$$
(I:z^i)+(z)=\mathfrak{a}_k
$$
for $i=ka, ka+1, \ldots, (k+1)a-1$, 
and $(I:z^{(n+1)a})+(z)=\mathfrak{a}_{n+1}$.  
\item[(2)] 
The following inclusions hold: 
$$
\mathfrak{a}_0 \subsetneq \mathfrak{a}_1 \subsetneq \mathfrak{a}_2 
\subsetneq \cdots \subsetneq \mathfrak{a}_{n} \subsetneq \mathfrak{a}_{n+1}.
$$
\end{itemize}
\end{lemma}

\begin{proof} 
Let $I_0, I_1,\ldots,I_n$ be the ideals defined in Proposition~\ref{Prop3-5}. 
Then it follows by Proposition~\ref{Prop3-5} that 
$$
I_k+(z)=\mathfrak{a}_k
$$
for $k=0,1,2,\ldots,n$. 
We divide the proof of (1) into three steps. 

{\em Step 1:} 
Since 
$$
(\tilde{p}_a, \tilde{p}_{a+1}, \ldots, \tilde{p}_{a+n-1}, z) = (p_a, p_{a+1}, \ldots, p_{a+n-1}, z), 
$$
it follows by Remark~\ref{Rem3-3} that  
$$
\{\tilde{p}_a, \tilde{p}_{a+1}, \ldots, \tilde{p}_{a+n-1}, z\}
$$
is also a regular sequence in $\tilde{R}$. 
Hence, using Proposition~\ref{Prop3-5}, 
one easily sees that 
$$
\begin{array}{rcl} 
I:z^i = I_0:z^i &=& (\tilde{p}_a, \tilde{p}_{a+1}, \ldots, \tilde{p}_{a+n-1}, z^{a}f^{(0)}):z^i  \\
        &=&  (\tilde{p}_a, \tilde{p}_{a+1}, \ldots, \tilde{p}_{a+n-1}, z^{a-i}f^{(0)})  \\
\end{array}
$$
for $i=0,1,2,\ldots,a-1$, and $I:z^a=I_1$. 
Thus $(I:z^i)+(z)=\mathfrak{a}_0$ for $i=0,1,2,\ldots,a-1$. 
This proves the assertion for the case $k=0$. 

{\em Step 2:} 
Since 
$$
(\tilde{p}_a, \tilde{p}_{a+1}, \ldots, \tilde{p}_{a+n-2}, z, f^{(0)}) = (p_a, p_{a+1}, \ldots, p_{a+n-2}, z, e_n), 
$$
it follows by Remark~\ref{Rem3-3} that  
$$
\{\tilde{p}_a, \tilde{p}_{a+1}, \ldots, \tilde{p}_{a+n-2}, z, f^{(0)}\}
$$
is also a regular sequence in $\tilde{R}$. 
Hence, using Proposition~\ref{Prop3-5}, 
one easily sees that 
$$
\begin{array}{rcl} 
I:z^{a+j} = I_1:z^j &=& (\tilde{p}_a, \tilde{p}_{a+1}, \ldots, \tilde{p}_{a+n-2}, z^{a}f^{(1)},f^{(0)}):z^j  \\
        &=&  (\tilde{p}_a, \tilde{p}_{a+1}, \ldots, \tilde{p}_{a+n-2}, z^{a-j}f^{(1)},f^{(0)})  \\
\end{array}
$$
for $j=0,1,2,\ldots,a-1$, and $I:z^{2a}=I_2$.   
Thus $(I:z^i)+(z)=\mathfrak{a}_1$ for $i=a,a+1,\ldots,2a-1$. 
This proves the assertion for the case $k=1$. 

{\em Step 3:} 
Let $k=2,3,\ldots,n$. 
In this step, we show that the equality
$$
(I:z^i)+(z)=\mathfrak{a}_k
$$
holds for $i=ka, ka+1, \ldots, (k+1)a-1$, 
assuming the equality $I:z^{ka}=I_k$. 
Noting that   
$$
\begin{array}{cl}
   & (\tilde{p}_a, \tilde{p}_{a+1}, \ldots, \tilde{p}_{a+n-k-1}, z, f^{(k-1)}, f^{(k-2)}, \ldots, f^{(1)}, f^{(0)}) \\
= & (p_a, p_{a+1}, \ldots, p_{a+n-k-1}, z, e_{n-k+1}, e_{n-k+2}, \ldots, e_{n-1}, e_n), 
\end{array}
$$
it follows by Remark~\ref{Rem3-3} that  
$$
\{\tilde{p}_a, \tilde{p}_{a+1}, \ldots, \tilde{p}_{a+n-k-1}, z, f^{(k-1)}, f^{(k-2)}, \ldots, f^{(1)}, f^{(0)}\}
$$
is also a regular sequence in $\tilde{R}$. 
Hence, using Proposition~\ref{Prop3-5}, 
one easily sees that 
$$
\begin{array}{rcl}
I:z^{ka+j}&=&I_k:z^j \\
&=&  (\tilde{p}_a, \tilde{p}_{a+1}, \ldots, \tilde{p}_{a+n-k-1}, z^{a}f^{(k)}, f^{(k-1)}, f^{(k-2)}, \ldots, f^{(1)}, f^{(0)}):z^j \\
&=&  (\tilde{p}_a, \tilde{p}_{a+1}, \ldots, \tilde{p}_{a+n-k-1}, z^{a-j}f^{(k)}, f^{(k-1)}, f^{(k-2)}, \ldots, f^{(1)}, f^{(0)}) \\
\end{array}
$$
for $j=0,1,2,\ldots,a-1$, and $I:z^{(k+1)a}=I_{k+1}$, where put $I_{n+1}=\tilde{R}$. 
Thus $(I:z^i)+(z)=\mathfrak{a}_k$ for $i=ka,ka+1,\ldots,(k+1)a-1$. 
This proves the assertion for the cases $k=2,3,\ldots,n$, by induction on $k$. 
Furthermore, one sees that 
$$
(I:z^{(n+1)a})+(z) = I_{n+1}+(z) = \mathfrak{a}_{n+1}. 
$$

(2) 
Since $(I:z^i)+(z) \subset (I:z^{i+1})+(z)$ for all $i\geq 0$, 
it follows by (1) that 
$$
\mathfrak{a}_0 \subset \mathfrak{a}_1 \subset \mathfrak{a}_2 
\subset \cdots \subset \mathfrak{a}_{n} \subset \mathfrak{a}_{n+1}. 
$$
Furthermore, since every ideal $\mathfrak{a}_k$ $(0\leq k\leq n)$ is generated by a regular sequence of length $n+1$, 
it follows by Remark~\ref{Rem:3-7} bellow that 
$$
\dim_K \tilde{R}/\mathfrak{a}_0 > \dim_K \tilde{R}/\mathfrak{a}_1 > 
\cdots > \dim_K \tilde{R}/\mathfrak{a}_n > \dim_K \tilde{R}/\mathfrak{a}_{n+1}=0
$$
as $K$-vector spaces. 
Hence we have
$$
\mathfrak{a}_0 \subsetneq \mathfrak{a}_1 \subsetneq \mathfrak{a}_2 
\subsetneq \cdots \subsetneq \mathfrak{a}_n \subsetneq \mathfrak{a}_{n+1}. 
$$
\end{proof}

\begin{remark}\label{Rem:3-7}
Let $\{f_1,f_2,\ldots,f_{n+1}\}$ be a regular sequence of homogeneous forms in $\tilde{R}$. 
Then one easily sees that 
$$
\dim_K \tilde{R}/(f_1,f_2,\ldots,f_{n+1}) = \prod_{i=1}^{n+1}\deg f_i
$$
as a $K$-vector space. 
Let $\{g_1,g_2,\ldots,g_{n}, g_{n+1}\}$ be a regular sequence of homogeneous forms in $\tilde{R}$. 
Assume that 
$$
\deg f_1 > \deg g_1, \deg f_i = \deg g_i \ (i=2,3,\ldots,n+1).  
$$
Then it follows that 
$$
\dim_K \tilde{R}/(f_1,f_2,\ldots,f_{n+1}) > \dim_K \tilde{R}/(g_1,g_2,\ldots,g_{n+1}).  
$$
\end{remark}

Now we can give a proof of the main result of this section.

\begin{proof}[Proof of Theorem \ref{Th3-1}]
(1) We note that $I=(\tilde{e}_1, \tilde{e}_2, \ldots, \tilde{e}_{n+1})$ and  $\tilde{e}_i ~(0\leq i\leq n+1)$ is defined as follows: 
$$
\prod_{i=1}^{n+1}(w-x_i)=\tilde{e}_0w^{n+1} + \tilde{e}_1w^n + \tilde{e}_2w^{n-1} + \cdots + \tilde{e}_{n}w + \tilde{e}_{n+1}, 
$$
where put $x_{n+1}=z$. 
Substituting $w=z$ in both sides of this equality, 
we have 
$$
\tilde{e}_0z^{n+1} + \tilde{e}_1z^n + \tilde{e}_2z^{n-1} + \cdots + \tilde{e}_{n}z + \tilde{e}_{n+1}=0.  
$$
Since $\tilde{e}_0=1$, 
it follows that $I=(\tilde{e}_1, \tilde{e}_2, \ldots, \tilde{e}_{n},z^{n+1})$. 
Furthermore, 
since 
$$
I:z^i=(\tilde{e}_1, \tilde{e}_2, \ldots, \tilde{e}_{n},z^{n+1-i})
$$ 
for $i=0,1,\ldots,n$ 
and $I^{n+1}=\tilde{R}$, 
it follows that 
$$
(I:z^i)+(z)=(e_1,e_2,\ldots,e_n,z)
$$ 
for $i=0,1,\ldots,n$ 
and $(I:z^{n+1})+(z)=\tilde{R}$. 
Thus 
$(A,\overline{z})$ has only one central simple module 
$$
U_1 \cong \tilde{R}/((\tilde{e}_1,\tilde{e}_2,\ldots,\tilde{e}_n,z) \cong R/(e_1,e_2,\ldots,e_n). 
$$

(2) Let $\mathfrak{a}_0, \mathfrak{a}_1, \ldots, \mathfrak{a}_{n+1}$ be the ideals defined in Lemma~\ref{KeyLem1}. 
Then it follows by Lemma~\ref{KeyLem1} that 
$$
U_j\cong \frac{\mathfrak{a}_{n+2-j}}{\mathfrak{a}_{n+1-j}}
$$
for $j=1,2,\ldots,n+1$. 
In particular, we have 
$$
U_1 \cong \frac{R}{(e_1,e_2,\ldots,e_n)}, 
$$ 
as desired. 
Next, for $j=2,3,\ldots,n+1$, 
a straightforward calculation leads to the following identities: 
$$
\frac{\mathfrak{a}_{n+2-j}}{\mathfrak{a}_{n+1-j}} \cong e_{j-1} \cdot \frac{\tilde{R}}{\mathfrak{a}_{n+1-j}} 
\cong \frac{\tilde{R}}{\mathfrak{a}_{n+1-j}:e_{j-1}}
$$
as $\tilde{R}$-modules. 
Hence, since  
$$
\begin{array}{rcl}
\mathfrak{a}_{n+1-j}:e_{j-1} &=& (p_a,p_{a+1},\ldots,p_{a+j-2},e_j,e_{j+1},\ldots,e_n,z):e_{j-1} \\
                                    &=& (p_{a-1},p_a,\ldots,p_{a+j-3},e_j,e_{j+1},\ldots,e_n,z) \\
\end{array}
$$ 
for $j=2,3,\ldots,n$, by Lemma~\ref{lem:colon} (1) stated bellow, 
we have 
$$
U_j \cong \frac{R}{(p_{a-1},p_a,\ldots,p_{a+j-3},e_j,e_{j+1},\ldots,e_n)}, 
$$
for $j=2,3,\ldots,n$, as desired. 
Furthermore, since 
$$
\begin{array}{rcl}
\mathfrak{a}_{0}:e_n &=& (p_a,p_{a+1},\ldots,p_{a+n-1},z):e_n \\
                           &=& (p_{a-1},p_a,\ldots,p_{a+n-2},z)\\
\end{array}
$$
by Lemma~\ref{lem:colon} (2), 
we have 
$$
U_{n+1} \cong \frac{R}{(p_{a-1},p_a,\ldots,p_{a+n-2})}, 
$$
as desired. 
\end{proof}

\begin{lemma}\label{lem:colon}
We use the same notation as Theorem~\ref{Th3-1}. 
\begin{itemize}
\item[(1)] 
Fix $s=0,1,2,\ldots,n-2$. 
Define the ideals $J$ and $J^\prime$ as follows: 
$$
\begin{array}{rcl}
J &=& (p_a,p_{a+1},\ldots,p_{a+s},e_{s+2},e_{s+3},\ldots,e_n,z), \\
J^\prime &=& (p_{a-1},p_a,\ldots,p_{a+s-1},e_{s+2},e_{s+3},\ldots,e_n,z). \\
\end{array}
$$ 
Then $J:e_{s+1}=J^\prime$. 
\item[(2)] 
The following holds: 
$$
(p_a,p_{a+1},\ldots,p_{a+n-1},z):e_n = (p_{a-1},p_a,\ldots,p_{a+n-2},z). 
$$ 
\end{itemize}
\end{lemma}

\begin{proof} 
(1) 
Substituting $k=a+3$ in Newton's identity, we have 
$$
e_0p_{a+s} + e_1p_{a+s-1} + \cdots + e_sp_{a} + e_{s+1}p_{a-1} \in (e_{s+2}, e_{s+3}, \ldots, e_n). 
$$
Hence $e_{s+1}p_{a-1} \in J$, i.e., $p_{a-1} \in J:e_{s+1}$.
Thus $J:e_{s+1} \supset J^\prime$. 

Let $f\in J:e_{s+1}$. 
Since $e_{s+1}f \in J$, 
$e_{s+1}f$ can be expressed as 
$$
e_{s+1}f = f_1p_a+f_2p_{a+1}+\cdots +f_{s+1}p_{a+s}+f_{s+2}e_{s+2}+f_{s+3}e_{s+3}+\cdots +f_ne_n+f_{n+1}z~(\star)
$$
for some $f_j\in\tilde{R}$. 
Furthermore, since 
$$
e_0p_{a+s} + e_1p_{a+s-1} + \cdots + e_sp_{a} + e_{s+1}p_{a-1} \in (e_{s+2}, e_{s+3}, \ldots, e_n), 
$$
we get the following equality: 
$$
e_0p_{a+s} + e_1p_{a+s-1} + \cdots + e_sp_{a} + e_{s+1}p_{a-1} - g_{s+2}e_{s+2} - \cdots - g_ne_n = 0  
$$
for some $g_j\in\tilde{R}$. 
Hence, substituting 
$$
p_{a+s}= -(e_1p_{a+s-1} + \cdots + e_{s+1}p_{a-1}  - g_{s+2}e_{s+2} - \cdots - g_ne_n)
$$
in the above equality~($\star$), we have 
$$
(f+f_{s+1}p_{a-1})e_{s+1} \in (p_a,p_{a+1},\ldots,p_{a+s-1},e_{s+2},e_{s+3},\ldots,e_n,z). 
$$
Since 
$$
\{p_a,p_{a+1},\ldots,p_{a+s-1},e_{s+1},e_{s+2},\ldots,e_n,z\}
$$
is a regular sequence by Remark~\ref{Rem3-3}, 
it follows that 
$$
(f+f_{s+1}p_{a-1}) \in (p_a,p_{a+1},\ldots,p_{a+s-1},e_{s+2},e_{s+3},\ldots,e_n,z), 
$$
and hence $f \in J^\prime$. 
Thus $J:e_{s+1} \subset J^\prime$. 

(2) 
Put $J=(p_a,p_{a+1},\ldots,p_{a+n-1},z)$ and $J^\prime = (p_{a-1},p_a,\ldots,p_{a+n-2},z)$. 
Newton's identity says that 
$$
e_0p_{a+n-1} + e_1p_{a+n-2} + \cdots + e_{n-1}p_a + e_np_{a-1}=0, 
$$
and hence $e_np_{a-1} \in J$, i.e., $p_{a-1} \in J:e_n$. 
Thus $J:e_n \supset J^\prime$. 

Let $f\in J:e_n$. 
Since $e_nf \in J$, 
$e_nf$ can be expressed as 
$$
e_nf = f_1p_a+f_2p_{a+1}+\cdots + f_np_{a+n-1} + f_{n+1}z
$$
for some $f_j\in\tilde{R}$. 
Hence, substituting 
$$
p_{a+n-1}= -(e_1p_{a+n-2} + \cdots + e_{n-1}p_a + e_np_{a-1})
$$
in the above equality, we have 
$$
(f+f_np_{a-1})e_n \in (p_a,p_{a+1},\ldots,p_{a+n-2},z). 
$$
Since 
$$
\{p_a,p_{a+1},\ldots,p_{a+n-2},e_n,z\}
$$
is a regular sequence by Remark~\ref{Rem3-3}, 
it follows that 
$$
(f+f_np_{a-1}) \in (p_a,p_{a+1},\ldots,p_{a+n-2},z). 
$$
and hence $f \in J^\prime$. 
Thus $J:e_{s+1} \subset J^\prime$. 
\end{proof}

\begin{remark}\label{Rem3-10} 
Let $a=2$ with the same notation as Theorem~\ref{Th3-1}. 
Then it follows that 
all central simple modules of $A=\tilde{R}/(\tilde{p}_2, \tilde{p}_3, \ldots, \tilde{p}_{n+2})$ 
are isomorphic to the complete intersection 
$R/(e_1, e_2, \ldots, e_n)$. 
\end{remark}

%
%

\section{Complete intersections defined by power sums and elementary symmetric polynomials}
\label{mixed}

Define $\tilde{e}_i \in \tilde{R}$ to be the elementary symmetric polynomial of degree $i$ {\em with the sign} 
so that 
$$
\prod_{i=1}^{n+1}(w-x_i)=\tilde{e}_0w^{n+1} + \tilde{e}_1w^n + \tilde{e}_2w^{n-1} + \cdots + \tilde{e}_{n}w + \tilde{e}_{n+1}, 
$$
where put $x_{n+1}=z$. 
Note that $\tilde{e}_0=1$, $\tilde{e}_{n+1}=-ze_n$ and $\tilde{e}_i=e_i-ze_{i-1}$ for $i=1,2,\ldots,n$. 
\medskip

The purpose of this section is to give a proof of the following theorem. 

\begin{theorem}\label{Th4-1} 
Let $a$ be an integer with $a\geq 2$ and let $I$ be an ideal of $\tilde{R}$ generated by 
$$
\{ 
\underbrace{\tilde{p}_a, \tilde{p}_{a+1}, \ldots, \tilde{p}_{a+b}}_{b+1 \geq 1},
\underbrace{ \tilde{e}_{b+2}, \tilde{e}_{b+3}, \ldots, \tilde{e}_{n+1} }_{n-b \geq 1}
 \}, 
$$ 
where $b$ is an integer such that $0\leq b \leq n-1$. Let $\bar{z}$ be the image of $z$ in $A=\tilde{R}/I$.
Then the Artinian complete intersection $(A, \bar{z})$ has $b+2$ central simple modules. 
The j-th central simple module is given by

$$
U_j \cong \displaystyle\frac{R}{(
\underbrace{p_{a-1}, p_a, \ldots, p_{a+j-3}}_{j-1}, 
\underbrace{e_j, e_{j+1}, \ldots, e_n}_{n-j+1}
)} 
$$
for $ \; j=1, 2, \ldots, b+2$.

\end{theorem}

We need some preparations for the proof of Theorem~\ref{Th4-1}. 

\begin{notation} 
Fix an integer $b$ with $0\leq b <n$. 
Let 
$$
g^{(0)}=g^{(0)}(z)=e_0z^b + e_1z^{b-1} + \cdots + e_{b-1}z + e_b \in \tilde{R} 
$$
and let
$$
g^{(k)}=g^{(k)}(z)=\frac{\partial^k}{\partial z^k}g^{(0)}(z)
$$
be the $k$-times partial derivation of $g^{(0)}$ 
for $k=1,2,\ldots,b$. 
Note that 
$$
f^{(0)}=z^{n-b}g^{(0)}+\sum_{b+1}^ne_iz^{n-i}
$$
(see Notation~\ref{def of f} for the definition of $f^{(0)}$). 
\end{notation}

Following is a key to the proof of Theorem~\ref{Th4-1}, and we can calculate the left child and the right child of our binary tree of complete intersections by using this proposition.

\begin{proposition}\label{Prop4-1}
Fix an integer $a\geq 2$ and an integer $b$ with $0\leq b <n$. 
With the same notation as above, 
define ideals $I_1, I_2, \ldots, I_{b+1}$ of $\tilde{R}$ as follows: 
$$ 
\begin{array}{rcl} 
I_1 & = & (\tilde{p}_a, \tilde{p}_{a+1}, \ldots, \tilde{p}_{a+b}, e_{b+1}, e_{b+2}, \ldots, e_n), \\
I_2 & = & (\tilde{p}_a, \tilde{p}_{a+1}, \ldots, \tilde{p}_{a+b-1}, g^{(0)}, e_{b+1}, e_{b+2}, \ldots, e_n), \\
I_3 & = & (\tilde{p}_a, \tilde{p}_{a+1}, \ldots, \tilde{p}_{a+b-2}, g^{(1)}, g^{(0)}, e_{b+1}, e_{b+2}, \ldots, e_n), \\
I_k & = & (\tilde{p}_a, \tilde{p}_{a+1}, \ldots, \tilde{p}_{a+b+1-k}, g^{(k-2)}, g^{(k-3)}, \ldots,g^{(0)}, e_{b+1}, e_{b+2}, \ldots, e_n) \\
\end{array}
$$
for $k=4,5,\ldots,b+1$. 
Then, for $k=1,2,\ldots,b+1$, 
the element $z^ag^{(k-1)}$ can replace the element $\tilde{p}_{a+b+1-k}$ 
in the ideal $I_k$ as a member of the generating set, i.e., 
$$
I_k = (\tilde{p}_a, \tilde{p}_{a+1}, \ldots, \tilde{p}_{a+b-k}, z^ag^{(k-1)}, g^{(k-2)}, g^{(k-3)}, \ldots, g^{(0)}, e_{b+1}, e_{b+2}, \ldots, e_n) 
$$ 
for $k=1,2,\ldots,b$, 
and $I_{b+1}=(z^ag^{(b)}, g^{(b-1)}, g^{(b-2)},\ldots,g^{(0)},e_{b+1},e_{b+2},\ldots,e_n)$. 

\end{proposition}

\begin{proof} 
We treat the three cases for $k=1,2,3$ separately. 
As one will see the proof for $k=3$ works for all other cases.

Let $k=1$. 
We show that 
$$
\tilde{p}_{a+b}-z^ag^{(0)} \in (\tilde{p}_a, \tilde{p}_{a+1}, \ldots, \tilde{p}_{a+b-1}, e_{b+1}, e_{b+2}, \ldots, e_n). 
$$
By Newton's identity, we have 
$$
e_0p_{a+b}+e_1p_{a+b-1}+ \cdots +e_bp_a \in (e_{b+1},e_{b+2},\ldots,e_n), 
$$
and hence 
$$
e_0p_{a+b}+e_1p_{a+b-1}+ \cdots +e_bp_a = h_{b+1}e_{b+1} + h_{b+2}e_{b+2} + \cdots + h_ne_n 
$$
for some $h_j \in R$. 
In this equality replace $p_{a+b-i}$ by $\tilde{p}_{a+b-i}-z^{a+b-i}$ for $i=0,1,\ldots,b$. 
Then we have 
$$
\sum_{i=0}^b e_i\tilde{p}_{a+b-i}=\sum_{i=0}^be_iz^{a+b-i} + \sum_{i=b+1}^nh_ie_i=z^ag^{(0)} +\sum_{i=b+1}^nh_ie_i. 
$$
Hence, since $e_0=1$, 
it follows that 
$$
\tilde{p}_{a+b}-z^ag^{(0)}=-\sum_{i=1}^b e_i\tilde{p}_{a+b-i}+\sum_{i=b+1}^nh_ie_i. 
$$
This proves the assertion for the case $k=1$.

Let $k=2$. 
We show that 
$$
\tilde{p}_{a+b-1}-2z^ag^{(1)}  \in (\tilde{p}_a, \tilde{p}_{a+1}, \ldots, \tilde{p}_{a+b-2}, g^{(0)}, e_{b+1}, e_{b+2}, \ldots, e_n). 
$$
Differentiating the equation $f^{(0)}=z^{n-b}g^{(0)}+e_{b+1}z^{n-b-1}+\cdots+e_{n-1}z+e_n$, 
we have 
$$ 
f^{(1)}  \equiv z^{n-b}g^{(1)}+(n-b)z^{n-b-1}g^{(0)} \mod (e_{b+1}, e_{b+2}, \ldots, e_n). 
$$
In the proof of Proposition~\ref{Prop3-5}, 
we have proved that 
$$
z^af^{(1)}=f^{(0)}\phi+{\bf e}{\bf P}{\bf u}. 
$$ 
Hence the following holds: 
$$
z^a\{z^{n-b}g^{(1)}+(n-b)z^{n-b-1}g^{(0)}\} \equiv  {\bf e}{\bf P}{\bf u} \mod (f^{(0)}, e_{b+1}, e_{b+2}, \ldots, e_n).
$$
Furthermore, 
comparing the terms of degree than $(n-b)$ or greater on both sides of the above equation, 
we have 
$$
z^a\{z^{n-b}g^{(1)}+(n-b)z^{n-b-1}g^{(0)}\} 
\equiv z^{n-b}{\bf e^\prime}{\bf P^\prime}{\bf u^\prime} \mod (f^{(0)}, e_{b+1}, e_{b+2}, \ldots, e_n), 
$$
where 
${\bf e^\prime}$, ${\bf P^\prime}$ and ${\bf u^\prime}$ 
are three matrices of the forms 
$$
{\bf e^\prime}=[e_{b-1}, e_{b-2}, \ldots, e_1, e_0], 
{\bf P^\prime}=
\left[
\begin{array}{ccccccc}
p_a          &      &           & \bigzerou  \\
p_{a+1}     & p_a &           &            \\
\vdots     & \ddots     & \ddots &           \\
p_{a+b-1}  & \cdots      &  p_{a+1}         & p_a   \\     
\end{array}
\right], 
{\bf u^\prime}=
\left[
\begin{array}{c}
1          \\
z    \\
\vdots     \\
z^{b-1}  \\     
\end{array}
\right].
$$
Since the ideal $(f^{(0)}, e_{b+1}, e_{b+2},\ldots,e_n)$ is contained in the ideal \\
 $(g^{(0)}, e_{b+1}, e_{b+2},\ldots,e_n)$, 
it follows that 
$$
z^{a+n-b}g^{(1)} \equiv z^{n-b}{\bf e^\prime}{\bf P^\prime}{\bf u^\prime} \mod (g^{(0)}, e_{b+1}, e_{b+2}, \ldots, e_n). 
$$
Thus, noting that $\{z, g^{(0)}, e_{b+1}, e_{b+2}, \ldots, e_n\}$ is a regular sequence, 
we have the following: 
$$
z^{a}g^{(1)} \equiv {\bf e^\prime}{\bf P^\prime}{\bf u^\prime} \mod (g^{(0)}, e_{b+1}, e_{b+2}, \ldots, e_n). 
$$
Put 
$$
\tilde{{\bf P}}^\prime=
\left[
\begin{array}{ccccccc}
\tilde{p}_a          &      &           & \bigzerou  \\
\tilde{p}_{a+1}     & \tilde{p}_a &           &            \\
\vdots     & \ddots     & \ddots &           \\
\tilde{p}_{a+b-1}  & \cdots      &  \tilde{p}_{a+1}         & \tilde{p}_a   \\     
\end{array}
\right], 
{\bf Z}^\prime=
\left[
\begin{array}{ccccccc}
1          &      &           & \bigzerou  \\
z     & 1 &           &            \\
\vdots     & \ddots     & \ddots &           \\
z^{b-1}  & \cdots      &  z        & 1   \\     
\end{array}
\right]. 
$$
Then, since ${\bf P}^\prime=\tilde{{\bf P}}^\prime-z^a{\bf Z}^\prime$, 
it follows that 
$$
z^ag^{(1)} \equiv {\bf e}^\prime\tilde{{\bf P}}^\prime{\bf u}^\prime-z^a{\bf e}^\prime{\bf Z}^\prime{\bf u}^\prime 
\mod (g^{(0)}, e_{b+1}, e_{b+2}, \ldots, e_n). 
$$
Hence, noting that ${\bf e}^\prime{\bf Z}^\prime{\bf u}^\prime=g^{(1)}$, 
we have 
$$
2z^ag^{(1)} \equiv {\bf e}^\prime\tilde{{\bf P}}^\prime{\bf u}^\prime 
\mod (g^{(0)}, e_{b+1}, e_{b+2}, \ldots, e_n). 
$$
This shows that $2z^ag^{(1)}$ is a linear combination of 
$$
\tilde{p}_a, \tilde{p}_{a+1}, \ldots, \tilde{p}_{a+b-1}, g^{(0)}, e_{b+1}, e_{b+2}, \ldots, e_n
$$
with coefficients in $\tilde{R}$. 
One sees that the coefficient of $\tilde{p}_{a+b-1}$ in the linear combination is $e_0=1$. 
This proves the assertion for the case $k=2$.

We proceed to a proof for the case $k=3$. 
We show that 
$$
\tilde{p}_{a+b-2}-\frac{3}{2}z^ag^{(2)}  \in (\tilde{p}_a, \tilde{p}_{a+1}, \ldots, \tilde{p}_{a+b-3}, g^{(1)}, g^{(0)}, e_{b+1}, e_{b+2}, \ldots, e_n). 
$$
Differentiating both sides of the equation 
$$
z^{a}g^{(1)} \equiv {\bf e}^\prime{\bf P}^\prime{\bf u}^\prime \mod (g^{(0)}, e_{b+1}, e_{b+2}, \ldots, e_n), 
$$
we have 
$$
z^ag^{(2)}+az^{a-1}g^{(1)} \equiv 
 {\bf e^\prime}{\bf P^\prime}({\bf u}^\prime)^{(1)} \mod (g^{(0)}, e_{b+1}, e_{b+2}, \ldots, e_n), 
$$
where $({\bf u}^\prime)^{(1)}$ is the transpose of the vector $[0, 1, 2z, 3z^2, \cdots, (b-1)z^{b-2}]$. 
Hence, substituting ${\bf P}^\prime=\tilde{{\bf P}}^\prime-z^a{\bf Z}^\prime$, 
we have 
$$
z^ag^{(2)}+az^{a-1}g^{(1)} \equiv 
 {\bf e}^\prime\tilde{{\bf P}}^\prime({\bf u}^\prime)^{(1)}-z^a{\bf e}^\prime{\bf Z}^\prime({\bf u}^\prime)^{(1)}
\mod (g^{(0)}, e_{b+1}, e_{b+2}, \ldots, e_n), 
$$
Furthermore, a straightforward calculation leads to the following equalities: 
$$
{\bf e}^\prime{\bf Z}^\prime({\bf u}^\prime)^{(1)} = \frac{1}{2}g^{(2)}. 
$$
Thus we get the equality 
$$
\frac{3}{2}z^ag^{(2)}\equiv 
 {\bf e}^\prime\tilde{{\bf P}}^\prime({\bf u}^\prime)^{(1)} \mod (g^{(1)}, g^{(0)}, e_{b+1}, e_{b+2}, \ldots, e_n). 
$$
This shows that $\frac{3}{2}z^ag^{(2)}$ is a linear combination of 
$$
\tilde{p}_a, \tilde{p}_{a+1}, \ldots, \tilde{p}_{a+b-2}, g^{(1)}, g^{(0)}, e_{b+1}, e_{b+2}, \ldots, e_n
$$
with coefficients in $\tilde{R}$. 
One sees that the coefficient of $\tilde{p}_{a+b-2}$ in the linear combination is $e_0=1$. 
This proves the assertion for the case $k=3$. 

Proof for $k>3$ is the same as the case $k=3$ with the obvious modification, 
except that we use the equalities stated in Lemma~\ref{Lem4-1} instead of the equality 
$
{\bf e}^\prime{\bf Z}^\prime({\bf u}^\prime)^{(1)}  = \frac{1}{2}g^{(2)}$. 
\end{proof}

\begin{lemma}\label{Lem4-1} 
With the same notation as the proof of Proposition~\ref{Prop4-1}, 
let $({\bf u}^\prime)^{(k)}$ be the $b$ dimensional column vector as follows: 
$$
({\bf u}^\prime)^{(k)} 
={}^{t}\!\left[\frac{\partial^k}{\partial z^k} 1, \frac{\partial^k}{\partial z^k} z,  \frac{\partial^k}{\partial z^k} z^2, 
\ldots, \frac{\partial^k}{\partial z^k} z^{b-1} \right], 
$$
where we denote the transpose of a vector by ${}^{t}[\ ]$. 
Then the following holds: 
$$
{\bf e}^\prime{\bf Z}^\prime\{({\bf u}^\prime)^{(k)}\} = \frac{1}{k+1}g^{(k+1)}
$$
for $k=2,3,\ldots,b-1$. 
\end{lemma}

\begin{proof}
In the proof of Lemma~\ref{Lem3-6}, replacing $n$ by $b$, we have
$$
\frac{\partial^k}{\partial z^k}({\bf Z}^\prime{\bf u}^\prime)=(k+1){\bf Z}^\prime\{({\bf u}^\prime)^{(k)}\}.
$$
Hence, noting that 
${\bf e}^\prime {\bf Z}^\prime {\bf u}^\prime = g^{(1)}$, 
we have
\begin{eqnarray*}
{\bf e}^\prime{\bf Z}^\prime\{({\bf u}^\prime)^{(k)}\} 
&=&\frac{1}{k+1}{\bf e}^\prime\frac{\partial^k}{\partial z^k}({\bf Z}^\prime{\bf u}^\prime)
=\frac{1}{k+1}\frac{\partial^k}{\partial z^k}({\bf e}^\prime{\bf Z}^\prime{\bf u}^\prime)\\
&=&\frac{1}{k+1}\frac{\partial^k}{\partial z^k}g^{(1)}
= \frac{1}{k+1}g^{(k+1)}.
\end{eqnarray*}
\end{proof}

\begin{lemma}\label{lem:4-2} 
With the same notation as Theorem~\ref{Th4-1}, 
define the ideal $\mathfrak{b}_0$ of $\tilde{R}$, 
$$
\mathfrak{b}_0=
 \left\{
 \begin{array}{ll}
(p_a,p_{a+1},\ldots,p_{a+b},e_{b+2},e_{b+3},\ldots,e_n,z)  & ($if$ \; \, b=0,1,\ldots,n-2), \\
(p_a, p_{a+1}, \ldots, p_{a+n-1}, z)  & ($if$ \; \, b=n-1).
\end{array}
\right.
$$
Then:
\begin{itemize}
\item[(1)] 
$(I:z^i)+(z)=\mathfrak{b}_0$ for $i=0,1,2,\ldots,n-b-1$. 
\item[(2)] 
$I:z^{n-b}=(\tilde{p}_a, \tilde{p}_{a+1}, \ldots, \tilde{p}_{a+b}, e_{b+1}, e_{b+2}, \ldots, e_n)$. 
\end{itemize}
\end{lemma}

\begin{proof} 
If $b=n-1$, then the assertion is clear. So we only prove the case when $b=0,1,\ldots,n-2$.
Since $\tilde{e}_{n+1}=-ze_{n}$ and $\tilde{e}_i=e_i-ze_{i-1}$ for $i=b+2, b+3,\ldots,n$, 
it is easy to check that the equality $I+(z)=\mathfrak{b}_0$ holds. 
Furthermore, we have   
$$
\begin{array}{cl}
  & (\tilde{p}_a, \tilde{p}_{a+1}, \ldots, \tilde{p}_{a+b}, \tilde{e}_{b+2}, \tilde{e}_{b+3}, \ldots, \tilde{e}_n, z) \\
= & (p_a, p_{a+1}, \ldots, p_{a+b}, e_{b+2}, e_{b+3}, \ldots, e_n, z). 
\end{array}
$$
Hence it follows by Remark~\ref{Rem3-3} that  
$$
\{\tilde{p}_a, \tilde{p}_{a+1}, \ldots, \tilde{p}_{a+b}, \tilde{e}_{b+2}, \tilde{e}_{b+3}, \ldots, \tilde{e}_n, z\}
$$
is a regular sequence in $\tilde{R}$. 
Thus, since $\tilde{e}_{n+1}=-ze_n$, 
we have  
$$
\begin{array}{rcl}
I:z & = & (\tilde{p}_a, \tilde{p}_{a+1}, \ldots, \tilde{p}_{a+b}, \tilde{e}_{b+2}, \tilde{e}_{b+3}, \ldots, \tilde{e}_n, ze_n):z \\
   & = & (\tilde{p}_a, \tilde{p}_{a+1}, \ldots, \tilde{p}_{a+b}, \tilde{e}_{b+2}, \tilde{e}_{b+3}, \ldots, \tilde{e}_n, e_n) \\
   & = & (\tilde{p}_a, \tilde{p}_{a+1}, \ldots, \tilde{p}_{a+b}, \tilde{e}_{b+2}, \tilde{e}_{b+3}, \ldots, \tilde{e}_{n-1}, ze_{n-1}, e_n). \\
\end{array}
$$
Hence $(I:z)+(z)=\mathfrak{b}_0$.

Consider the ideals $\{J_j\}$ as follows: 
$$
J_j=(\tilde{p}_a, \tilde{p}_{a+1}, \ldots, \tilde{p}_{a+b}, 
\tilde{e}_{b+2}, \tilde{e}_{b+3}, \ldots, \tilde{e}_{n-j}, ze_{n-j}, e_{n-j+1},e_{n-j+2},\ldots,e_n )
$$
for $j=1,2,\ldots,n-b-2$, and 
$$
J_{n-b-1}=(\tilde{p}_a, \tilde{p}_{a+1}, \ldots, \tilde{p}_{a+b}, ze_{b+1}, e_{b+2}, e_{b+3},\ldots, e_n). 
$$
Similarly, one sees that 
$$
\{\tilde{p}_a, \tilde{p}_{a+1}, \ldots, \tilde{p}_{a+b}, 
\tilde{e}_{b+2}, \tilde{e}_{b+3}, \ldots, \tilde{e}_{n-j}, z, e_{n-j+1},e_{n-j+2},\ldots,e_n\}
$$
is also a regular sequence in $\tilde{R}$. 
Hence we have 
$$
J_j:z=J_{j+1}
$$
for $j=1,2,\ldots,n-b-1$, where put 
$$
J_{n-b}=(\tilde{p}_a, \tilde{p}_{a+1}, \ldots, \tilde{p}_{a+b}, e_{b+1}, e_{b+2}, \ldots, e_n). 
$$
Therefore, we see that $I:z^i=J_i$ for $i=1,2,\ldots,n-b$. 
In particular, the equality (2) holds. 
Furthermore, we have the equalities of (1):  
$$
(I:z^i)+(z) = J_i+(z) = \mathfrak{b}_0
$$
for $i=2,3,\ldots,n-b-1$. 
\end{proof}

\begin{lemma}\label{KeyLem2}
With the same notation as Theorem~\ref{Th4-1}, 
define the ideals $\mathfrak{b}_1, \mathfrak{b}_2, \ldots, \mathfrak{b}_{b+2}$ of $\tilde{R}$ as follows: 
$$
\begin{array}{rcl}
\mathfrak{b}_1 &=& (p_a,p_{a+1},\ldots,p_{a+b-1},e_{b+1},e_{b+2},\ldots,e_n,z),  \\
\mathfrak{b}_2 &=& (p_a,p_{a+1},\ldots,p_{a+b-2},e_{b},e_{b+1},\ldots,e_n,z),  \\
\mathfrak{b}_k &=& (p_a,p_{a+1},\ldots,p_{a+b-k},e_{b+2-k},e_{b+3-k},\ldots,e_n,z)
\end{array}
$$ 
for $k=3,4,\ldots,b$, $\mathfrak{b}_{b+1} = (e_1,e_2,\ldots,e_n,z)$ and $\mathfrak{b}_{b+2}=\tilde{R}$. 
Put 
$$
c_k=n-b+(k-1)a
$$ 
for $k=1,2,\ldots,b+2$. 
\begin{itemize}
\item[(1)] 
For $k=1,2,\ldots,b+1$, 
we have the following equalities: 
$$
(I:z^i)+(z)=\mathfrak{b}_k
$$ 
for $i=c_k, c_k+1,\ldots,c_{k+1}-1$, 
and $(I:z^{c_{b+2}})+(z)=\mathfrak{b}_{b+2}$. 
\item[(2)] 
The following inclusions hold: 
$$
\mathfrak{b}_0 \subsetneq \mathfrak{b}_1 \subsetneq \mathfrak{b}_2 
\subsetneq \cdots \subsetneq \mathfrak{b}_{b+1} \subsetneq \mathfrak{b}_{b+2},  
$$
where $\mathfrak{b}_0$ is the ideal defined in Lemma~\ref{lem:4-2}. 
\end{itemize} 
\end{lemma}

\begin{proof} 
Let $I_1, I_2,\ldots,I_{b+1}$ be the ideals defined in Proposition~\ref{Prop4-1}. 
Then it follows by Proposition~\ref{Prop4-1} that 
$$
I_k+(z)=\mathfrak{b}_k
$$
for $k=1,2,\ldots,b+1$. 
We divide the proof of (1) into three steps. 

{\em Step 1:} 
Since 
\begin{eqnarray*}
&&(\tilde{p}_a, \tilde{p}_{a+1}, \ldots, \tilde{p}_{a+b-1}, z, e_{b+1}, e_{b+2}, \ldots, e_n)  \\
&=& (p_a, p_{a+1}, \ldots, p_{a+b-1}, z, e_{b+1}, e_{b+2}, \ldots, e_n), 
\end{eqnarray*}
it follows by Remark~\ref{Rem3-3} that  
$$
\{\tilde{p}_a, \tilde{p}_{a+1}, \ldots, \tilde{p}_{a+b-1}, z, e_{b+1}, e_{b+2}, \ldots, e_n\}
$$
is also a regular sequence in $\tilde{R}$. 
Furthermore it follows by Lemma~\ref{lem:4-2} (2) that $I:z^{n-b}=I_1$. 
Hence, using Proposition~\ref{Prop4-1}, one easily sees that 
$$
\begin{array}{rcl} 
I:z^{n-b+j}= I_1:z^j &=& (\tilde{p}_a, \tilde{p}_{a+1}, \ldots, \tilde{p}_{a+b-1}, z^{a}g^{(0)},e_{b+1},e_{b+2},\ldots,e_n):z^j  \\
        &=&  (\tilde{p}_a, \tilde{p}_{a+1}, \ldots, \tilde{p}_{a+b-1}, z^{a-j}g^{(0)},e_{b+1},e_{b+2},\ldots,e_n) \\
\end{array}
$$
for $j=0,1,2,\ldots,a-1$, and $I:z^{c_2}=I:z^{n-b+a}=I_1:z^{a}=I_2$. 
Thus $(I:z^i)+(z)=\mathfrak{b}_1$ for $i=c_1,c_1+1,\ldots,c_2-1$. 
This proves the assertion for the case $k=1$.

{\em Step 2:} 
Since 
\begin{eqnarray*}
&&(\tilde{p}_a, \tilde{p}_{a+1}, \ldots, \tilde{p}_{a+b-2}, z, g^{(0)},e_{b+1},e_{b+2},\ldots,e_n) \\
&=&(p_a, p_{a+1}, \ldots, p_{a+b-2}, z, e_b,e_{b+1},\ldots,e_n), 
\end{eqnarray*}
it follows by Remark~\ref{Rem3-3} that  
$$
\{\tilde{p}_a, \tilde{p}_{a+1}, \ldots, \tilde{p}_{a+b-2}, z, g^{(0)},e_{b+1},e_{b+2},\ldots,e_n\}
$$
is also a regular sequence in $\tilde{R}$. 
Hence, using Proposition~\ref{Prop4-1},  
one easily sees that 
$$
\begin{array}{rcl} 
I:z^{c_2+j} = I_2:z^j &=& 
(\tilde{p}_a, \tilde{p}_{a+1}, \ldots, \tilde{p}_{a+b-2}, z^{a}g^{(1)},g^{(0)},e_{b+1},e_{b+2},\ldots,e_n):z^j  \\
        &=&  (\tilde{p}_a, \tilde{p}_{a+1}, \ldots, \tilde{p}_{a+b-2}, z^{a-j}g^{(1)},g^{(0)},e_{b+1},e_{b+2},\ldots,e_n) \\
\end{array}
$$
for $j=0,1,2,\ldots,a-1$, and $I:z^{c_3}=I_3$.   
Thus $(I:z^i)+(z)=\mathfrak{b}_2$ for $i=c_2,c_2+1,\ldots,c_3-1$. 
This proves the assertion for the case $k=2$.

{\em Step 3:} 
Let $k=3,4,\ldots,b+1$. 
In this step, we show that the equality
$$
(I:z^i)+(z)=\mathfrak{b}_k
$$
holds for $i=c_k, c_k+1, \ldots, c_{k+1}-1$, 
assuming the equality $I:z^{c_k}=I_k$. 
Noting that   
$$
\begin{array}{cl}
   & (\tilde{p}_a, \tilde{p}_{a+1}, \ldots, \tilde{p}_{a+b-k}, z, g^{(k-2)}, g^{(k-3)}, \ldots, g^{(1)}, g^{(0)},e_{b+1},e_{b+2},\ldots,e_n) \\
= & (p_a, p_{a+1}, \ldots, p_{a+b-k}, z, e_{b-k+2}, e_{b-k+3}, \ldots, e_{n-1}, e_n), 
\end{array}
$$
it follows by Remark~\ref{Rem3-3} that  
$$
\{(\tilde{p}_a, \tilde{p}_{a+1}, \ldots, \tilde{p}_{a+b-k}, z, g^{(k-2)}, g^{(k-3)}, \ldots, g^{(1)}, g^{(0)},e_{b+1},e_{b+2},\ldots,e_n\}
$$
is also a regular sequence in $\tilde{R}$. 
Hence, using Proposition~\ref{Prop4-1},  
one sees that 
\begin{eqnarray*}
I:z^{c_k+j} &=& I_k:z^j  \\
&=&  (\tilde{p}_a, \tilde{p}_{a+1}, \ldots, \tilde{p}_{a+b-k}, z^{a}g^{(k-1)}, g^{(k-2)}, g^{(k-3)},\\
&& \hspace{137pt} \ldots, g^{(1)}, g^{(0)}, e_{b+1},e_{b+2},\ldots,e_n):z^j  \quad \\
&=&  (\tilde{p}_a, \tilde{p}_{a+1}, \ldots, \tilde{p}_{a+b-k}, z^{a-j}g^{(k-1)}, g^{(k-2)}, g^{(k-3)},\\
&& \hspace{137pt} \ldots, g^{(1)}, g^{(0)},e_{b+1},e_{b+2},\ldots,e_n) \\[-20pt]
\end{eqnarray*}
for $j=0,1,2,\ldots,a-1$, and $I:z^{c_{k+1}}=I_{k+1}$, where put $I_{b+2}=\tilde{R}$. 
Thus $(I:z^i)+(z)=\mathfrak{b}_k$ for $i=c_k, c_k+1,\ldots, c_{k+1}-1$. 
This proves the assertion for the cases $k=3,4,\ldots,b+1$, by induction on $k$. 
Furthermore one easily sees that 
 $(I:z^{c_{b+2}})+(z)=\mathfrak{b}_{b+2}$.

(2) 
Since $(I:z^i)+(z) \subset (I:z^{i+1})+(z)$ for all $i\geq 0$, 
it follows by (1) that 
$$
\mathfrak{b}_0 \subset \mathfrak{b}_1 \subset \mathfrak{b}_2 
\subset \cdots \subset \mathfrak{b}_{b+1} \subset \mathfrak{b}_{b+2}.  
$$
Furthermore, since every ideal $\mathfrak{b}_k$ $(0\leq k\leq b+1)$ is generated by a regular sequence of length $n+1$, 
it follows by Remark~\ref{Rem:3-7} that 
$$
\dim_K \tilde{R}/\mathfrak{b}_0 > \dim_K \tilde{R}/\mathfrak{b}_1 > 
\cdots > \dim_K \tilde{R}/\mathfrak{b}_{b+1} > \dim_K \tilde{R}/\mathfrak{b}_{b+2}=0
$$
as $K$-vector spaces. 
Hence we have
$$
\mathfrak{b}_0 \subsetneq \mathfrak{b}_1 \subsetneq \mathfrak{b}_2 
\subsetneq \cdots \subsetneq \mathfrak{b}_{b+1} \subsetneq \mathfrak{b}_{b+2}. 
$$
\end{proof}

Now we can give a proof of the main result of this section. 

\begin{proof}[Proof of Theorem \ref{Th4-1}]
Let $\mathfrak{b}_0, \mathfrak{b}_1, \ldots, \mathfrak{b}_{b+2}$ be the ideals defined in Lemmas~\ref{lem:4-2} and~\ref{KeyLem2}. 
Then it follows by Lemmas~\ref{lem:4-2} and~\ref{KeyLem2} that 
$$
U_j\cong \frac{\mathfrak{b}_{b+3-j}}{\mathfrak{b}_{b+2-j}}
$$
for $j=1,2,\ldots,b+2$. 
In particular, we have 
$$
U_1 \cong \frac{R}{(e_1,e_2,\ldots,e_n)}, 
$$ 
as desired. 
Next, for $j=2,3,\ldots,b+2$, 
a straightforward calculation leads to the following identities: 
$$
\frac{\mathfrak{b}_{b+3-j}}{\mathfrak{b}_{b+2-j}} \cong e_{j-1} \cdot \frac{\tilde{R}}{\mathfrak{b}_{b+2-j}} 
\cong \frac{\tilde{R}}{\mathfrak{b}_{b+2-j}:e_{j-1}}
$$
as $\tilde{R}$-modules. 
Hence, since  
$$
\begin{array}{rcl}
\mathfrak{b}_{b+2-j}:e_{j-1} &=& (p_a,p_{a+1},\ldots,p_{a+j-2},e_j,e_{j+1},\ldots,e_n,z):e_{j-1} \\
                                    &=& (p_{a-1},p_a,\ldots,p_{a+j-3},e_j,e_{j+1},\ldots,e_n,z) \\
\end{array}
$$ 
for $j=2,3,\ldots,b+1$, by Lemma~\ref{lem:colon}, we have 
$$
U_j \cong \frac{R}{(p_{a-1},p_a,\ldots,p_{a+j-3},e_j,e_{j+1},\ldots,e_n)} 
$$
for $j=2,3,\ldots,b+1$, as desired. 
Similarly, it is easy to show that 
$$
U_{b+2} \cong \frac{R}{(p_{a-1}, p_a, \ldots, p_{a+b-1}, e_{b+2}, e_{b+3}, \ldots, e_n)}
$$ 
when $b<n-1$. 
If $b=n-1$, then we have
$$
\begin{array}{rcl}
\mathfrak{b}_{0}:e_n &=& (p_a,p_{a+1},\ldots,p_{a+n-1},z):e_n \\
                           &=& (p_{a-1},p_a,\ldots,p_{a+n-2},z) \\
\end{array}
$$
by Lemma~\ref{lem:colon}. 
Hence it follows that  
$$
U_{b+2} \cong \frac{R}{(p_{a-1},p_a,\ldots,p_{a+n-2})} .
$$
\end{proof}

\begin{remark}\label{Rem4-7} 
\begin{itemize} 
\item[(1)] 
Suppose that $a=2$ in Theorem~\ref{Th4-1}. That is,
let
 $$A=\tilde{R}/
(\tilde{p}_2, \tilde{p}_3, \ldots, \tilde{p}_{b+2}, \tilde{e}_{b+2}, \tilde{e}_{b+3}, \ldots, \tilde{e}_{n+1}). 
$$
Then, by Theorem~\ref{Th4-1}, 
all central simple modules of $A$ are isomorphic to the complete intersection 
$R/(e_1, e_2, \ldots, e_n)$. 
\item[(2)]
Let 
$A=\tilde{R}/(\tilde{p}_a, \tilde{p}_{a+1}, \ldots, \tilde{p}_{a+n})$ and 
$B=\tilde{R}/(\tilde{p}_a, \tilde{p}_{a+1}, \ldots, \tilde{p}_{a+n-1}, \tilde{e}_{n+1})$. 
Then it follows by Theorems~\ref{Th3-1} and~\ref{Th4-1} that 
$A$ and $B$ have the same central simple modules. 
\end{itemize}
\end{remark}

%
%

\section{The strong Lefschetz property for complete intersections}
\label{slp}

In this section, using the results in previous sections 3 and 4, we observe the strong
Lefschetz property for complete intersections in Theorem~\ref{thirdTh} . 

\begin{notation}
Let $n \geq 1$, $a \geq 2$ and $m=1,2,\ldots,n$ be integers. 
Let $R=K[x_1,x_2,\ldots,x_n]$. 
We define Artinian complete intersection $K$-algebras
$$
A_n(a,m)=R/(p_a,p_{a+1}\ldots,p_{a+m-1},e_{m+1},e_{m+2},\ldots,e_{n})
$$
for $m=1,2,\ldots,n-1$, 
$A_n(a,n)=R/(p_a,p_{a+1}\ldots,p_{a+n-1})$ and $A_n(1,n)=R/(e_1,e_2,\ldots,e_n)$. 
We denote
$$ 
\mathcal A_n =\{A_n(a,m) \mid a \geq 2,m=1,2,\ldots,n \} \cup \{A_n(1,n)\}
$$
for $n=1,2,\ldots$. 
\end{notation}

Applying Theorem~\ref{Th3-1} and Theorem~\ref{Th4-1}, 
let's see what happens when we
continue to consider central simple modules, as illustrated in the following example.

\begin{example}\label{Ex-diagram} 
In the diagram bellow, 
we simply denote ${}_na_m=A_n(a,m)$, 
and depict the derived central simple modules of $(A_5(3,3),x_5)$ 
and  
$(A_5(8,3),x_5)$. 
If two labels are connected by an arrow like ${}_na_m \rightarrow {}_{n-1}(a-1)_{m^{\prime}}$ 
with the terminal label ${}_{n-1}(a-1)_{m^{\prime}}$ in $U_i$-column, 
then $A_{n-1}(a-1,m^{\prime})$ is the $i$-th central simple module of $(A_n(a,m), x_n)$. 
For example: 
\begin{itemize} 
\item 
It follows by Theorem~\ref{Th4-1} that 
$(A_5(8,3),x_5)$ have the following four central simple modules 
$$
U_1 \cong A_4(1,4), U_2 \cong A_4(7,1), U_3 \cong A_4(7,2), U_4 \cong A_4(7,3),   
$$
i.e., 
$$
{}_58_3 \rightarrow {}_41_4, {}_58_3 \rightarrow {}_47_1, {}_58_3 \rightarrow {}_47_2, {}_58_3 \rightarrow {}_47_3. 
$$
\item 
Similarly, it follows by Theorem~\ref{Th3-1}(2) that 
$(A_3(6,3),x_3)$ have the following three central simple modules 
$$
U_1 \cong A_2(1,2), U_2 \cong A_2(5,1), U_3 \cong A_2(5,2), 
$$
i.e., 
$$
{}_36_3 \rightarrow {}_21_2, {}_36_3 \rightarrow {}_25_1, {}_36_3 \rightarrow {}_25_2. 
$$ 
\item 
Furthermore, Theorem~\ref{Th3-1}(1) gives the following three arrows, 
$$
{}_41_4 \rightarrow {}_31_3 \rightarrow {}_21_2 \rightarrow {}_11_1. 
$$ 
\end{itemize}

\[\begin{tikzcd}
	{\mathcal{A}_5} & {{}_53_{3}} &&&&&& {{}_58_{3}} \\
	{\mathcal{A}_4} & {{}_42_{3}} & {{}_42_{2}} & {{}_42_{1}} & {{}_41_{4}} & {{}_47_{1}} & {{}_47_{2}} & {{}_47_{3}} \\
	{\mathcal{A}_3} & {{}_31_{3}} & {{}_31_{3}} & {{}_31_{3}} & {{}_31_{3}} & {{}_36_{1}} & {{}_36_{2}} & {{}_36_{3}} \\
	{\mathcal{A}_2} &&&& {{}_21_{2}} & {{}_25_{1}} & {{}_25_{2}} \\
	{\mathcal{A}_1} &&&& {{}_11_{1}} & {{}_14_{1}} \\
	& {U_{4}} & {U_{3}} & {U_{2}} & {U_{1}} & {U_{2}} & {U_{3}} & {U_{4}}
	\arrow[from=1-8, to=2-5]
	\arrow[from=1-8, to=2-6]
	\arrow[from=1-8, to=2-7]
	\arrow[from=1-8, to=2-8]
	\arrow[from=2-5, to=3-5]
	\arrow[from=3-5, to=4-5]
	\arrow[from=4-5, to=5-5]
	\arrow[from=2-6, to=3-6]
	\arrow[from=3-6, to=4-6]
	\arrow[from=4-6, to=5-6]
	\arrow[from=2-7, to=3-7]
	\arrow[from=3-7, to=4-7]
	\arrow[from=2-8, to=3-8]
	\arrow[from=2-6, to=3-5]
	\arrow[from=2-7, to=3-5]
	\arrow[from=2-7, to=3-6]
	\arrow[from=3-6, to=4-5]
	\arrow[from=3-7, to=4-5]
	\arrow[from=3-7, to=4-6]
	\arrow[from=4-6, to=5-5]
	\arrow[from=4-7, to=5-5]
	\arrow[from=4-7, to=5-6]
	\arrow[from=3-8, to=4-5]
	\arrow[from=3-8, to=4-6]
	\arrow[from=3-8, to=4-7]
	\arrow[from=2-8, to=3-5]
	\arrow[from=2-8, to=3-6]
	\arrow[from=2-8, to=3-7]
	\arrow[from=1-2, to=2-2]
	\arrow[from=1-2, to=2-3]
	\arrow[from=1-2, to=2-4]
	\arrow[from=1-2, to=2-5]
	\arrow[from=2-2, to=3-2]
	\arrow[from=2-2, to=3-3]
	\arrow[from=2-2, to=3-4]
	\arrow[from=2-2, to=3-5]
	\arrow[from=2-3, to=3-3]
	\arrow[from=2-3, to=3-4]
	\arrow[from=2-3, to=3-5]
	\arrow[from=2-4, to=3-4]
	\arrow[from=2-4, to=3-5]
	\arrow[from=3-4, to=4-5]
	\arrow[from=3-3, to=4-5]
	\arrow[from=3-2, to=4-5]
\end{tikzcd}\]

As noted in Remark~\ref{Rem4-7}, we can see the following:
\begin{itemize}
\item
The central simple modules of $(A_4(2,m),x_4)$ are all isomorphic to $A_3(1,3)$ for $m=1,2,3$.
\item
The sets of central simple modules of $(A_3(6,2),x_3)$ and $(A_3(6,3),x_3)$ coincide. 
\end{itemize}
\end{example}

\begin{theorem}\label{thirdTh} 
Every Artinian complete intersection $K$-algebra 
defined in Theorem~\ref{Th3-1} and Theorem~\ref{Th4-1} has the SLP. 
\end{theorem}

\begin{proof}
It is enough to show that all elements of $\mathcal A_n$ 
have the SLP. We prove this by induction on $n$. 
If $n=1$, then it is clear that all elements of $\mathcal A_1$ 
have the SLP, since any element in $\mathcal A_1$ is of the form $k[x_1]/(x_1^i)$ for some $i \geq 1$. 
Let $n > 1$ and take an elemenet $\xi\in \mathcal A_n$. Then Theorem~\ref{Th3-1} and Theorem~\ref{Th4-1} assert that all central simple modules of $\xi$ are contained in $\mathcal A_{n-1}$. Using the induction hypothesis, all central simple modules of $\xi$ have the SLP.  This completes the proof thanks to Theorem~\ref{csm}. 
\end{proof}

\begin{remark}\label{Rem5-1} 
We can summarize our results in terms of "binary tree".
Consider the following families of ideals in $K[x_1, x_2, \ldots, x_n]$, 
$\mathcal F_{n} = \mathcal B_{n} \cup \mathcal C_{n}$ for all $n \geq 1$, where 
$$
\begin{array}{rcl} 
\mathcal B_{n} & = &
\{
(p_a,p_{a+1}, \ldots, p_{a+n-1}) : x_{n}^{i} \mid a \geq 1, 0 \leq i \leq an-1
\},\\
\mathcal C_{n} & = &
\{
(p_a,p_{a+1}, \ldots, p_{a+b-1}, e_{b+1}, e_{b+2}, \ldots, e_{n}) : x_{n}^{i} \mid \\
&  & \hspace{98pt} a \geq 2,1 \leq b \leq n-1, 0 \leq i \leq (a-1)b+n-1
\}.\\
\end{array}
$$
Then, by our results, it is easy to see that the family $\mathcal F = \bigcup _{n \geq 1}\mathcal F_{n}$ is a binary tree of complete intersections with the SLP. 
\end{remark}

\section*{Acknowledgement}
This work was supported by JSPS KAKENHI Grant Number JP20K03508.
 Our proof technique to use a binary tree of ideals was inspired by the definition of the ``principal radical system'' introduced by Hochster—Eagon \cite{hochster_eagon}.   This idea was used 
  inevitably in the papers \cite{mcdaniel_watana_1} and \cite{mcdaniel_watana_2}. The third author learned, from Larry Smith, that the idea was originally used by K. Oka in his paper~\cite{oka} (see Lam and Reyes~\cite{lam_reyes}). 





\vspace{30pt}
\today
\end{document}